\documentclass{amsart}
\usepackage{enumerate}

\theoremstyle{plain}
\newtheorem{thm}{Theorem}[section]
\newtheorem{prop}[thm]{Proposition}

\newtheorem{lem}[thm]{Lemma}

\newtheorem{defi}[thm]{Definition}
\newtheorem{ex}[thm]{Example}

\begin{document}

\title{Connes integration formula for the noncommutative plane}

\author[]{F. Sukochev}
\address{School of Mathematics and Statistics, University of New South Wales, Kensington,  2052, Australia}
\email{f.sukochev@unsw.edu.au}

\author[]{D. Zanin}
\address{School of Mathematics and Statistics, University of New South Wales, Kensington,  2052, Australia}
\email{d.zanin@unsw.edu.au}

\begin{abstract} Our aim is to prove the integration formula on the noncommutative (Moyal) plane in terms of singular traces {\it a la} Connes.
\end{abstract}

\maketitle

\section{Introduction}

Let $M$ be a compact Riemannian manifold. The following formula can be found in p.~34 in \cite{BF} and in Corollary 7.21 in \cite{GVF}.
\begin{equation}\label{original}
{\rm Tr}_{\omega}(M_f(1-\Delta)^{-\frac{d}{2}})=\int_M f d{\rm vol},\quad f\in C^{\infty}(M).
\end{equation}
Here, $M_f$ is the multiplication operator, $\Delta$ is the Hodge-Laplacian operator on $L_2(M,{\rm vol})$ and ${\rm Tr}_{\omega}$ is the Dixmier trace on the ideal $\mathcal{L}_{1,\infty}$ (see Section \ref{prelims}). Also, Corollary 7.22 in \cite{GVF} wrongly extends this result to $f \in L_1(M,{\rm vol})$ (in fact, $f\in L_2(M,{\rm vol})$ is the necessary and sufficient condition for this formula to hold; see \cite{LPS} or the book \cite{LSZ} for detailed proofs).

According to \cite{BF}, formula \eqref{original} \lq\lq led Connes to introduce the Dixmier trace as the correct operator theoretical substitute for integration of infinitesimals of order one in non-commutative geometry.\rq\rq~It appears suitable to refer to \eqref{original} and similar results as the \lq\lq Connes Integration Formula\rq\rq.

Compactness of the (resolvent of the) Hodge-Dirac operator plays a crucial role in the proofs of Connes Integration Formula for unital spectral triples (see \cite{BF} and \cite{GVF}). For non-unital spectral triples (including non-compact manifolds), the proofs become radically harder. Even the case of the simplest non-compact manifold $\mathbb{R}^d$ required a substantial effort and the first reasonable answer was very recently given in \cite{KLPS} (see the book \cite{LSZ} for detailed proofs).

In this paper, we investigate the validity of Connes Integration Formula for the noncommutative (Moyal) plane $\mathbb{R}^d_{\theta}$ (here, $\theta$ is a non-degenerate antisymmetric matrix). Earlier attempts in this direction can be found in \cite{gayral-moyal} (see Proposition 4.17 there), \cite{CGRS1} and \cite{CGRS2}. We substantially strengthen corresponding results from these papers and present a completely different approach to Connes Integration Formula. The novelty of our approach is in the consistent use of Cwikel estimates for the noncommutative plane (obtained in a recent paper \cite{LeSZ-cwikel}) --- see Section \ref{prelims}. 

Our main result is the following theorem.

\begin{thm}\label{cif ncplane} If $x\in W^{d,1}(\mathbb{R}^d_{\theta}),$ then $x(1-\Delta)^{-\frac{d}{2}}\in\mathcal{L}_{1,\infty}$ and
$$\varphi(x(1-\Delta)^{-\frac{d}{2}})=\tau_{\theta}(x)$$
for every normalised continuous trace $\varphi$ on $\mathcal{L}_{1,\infty}.$
\end{thm}
Here, $W^{d,1}(\mathbb{R}^d_{\theta})$ is a Sobolev space on $\mathbb{R}^d_{\theta}$ and $\tau_{\theta}$ is the faithful normal semifinite trace on $L_{\infty}(\mathbb{R}^d_{\theta}).$

Section~\ref{prelims} involves the preliminaries necessary to prove~Theorem \ref{cif ncplane}. In Section \ref{proportionality section}, we prove that
$$\varphi(x(1-\Delta)^{-\frac{d}{2}})=c_{\varphi}\tau_{\theta}(x),\quad x\in W^{d,1}(\mathbb{R}^d_{\theta}),$$
for every normalised trace on $\mathcal{L}_{1,\infty}.$ In Section \ref{measurability section}, we construct {\it one particular} $x\in W^{d,1}(\mathbb{R}^d_{\theta})$ such that $\varphi(x(1-\Delta)^{-\frac{d}{2}})$ does not depend on the choice of a normalised continuous trace $\varphi.$ The combination of these results yield Theorem \ref{cif ncplane}.

\section{Preliminaries}\label{prelims}

\subsection{General notation}

Fix throughout a separable infinite dimensional Hilbert space $H.$ We let $\mathcal{L}(H)$ denote the algebra of all bounded operators on $H.$ For a compact operator $T$ on $H,$ let $\mu(k,T)$ denote $k-$th largest singular value (these are the eigenvalues of $|T|$). The sequence $\mu(T)=\{\mu(k,T)\}_{k\geq0}$ is referred to as to the singular value sequence of the operator $T.$ The standard trace on $\mathcal{L}(H)$ is denoted by ${\rm Tr}.$

Fix an orthonormal basis in $H$ (the particular choice of a basis is inessential). We identify the algebra $l_{\infty}$ of bounded sequences with the subalgebra of all diagonal operators with respect to the chosen basis. For a given sequence $\alpha\in l_{\infty},$ we denote the corresponding diagonal operator by ${\rm diag}(\alpha).$

\subsection{Schatten ideals $\mathcal{L}_p$ and $\mathcal{L}_{p,\infty},$ $p>0$}

For every $p>0,$ we set 
$$\mathcal{L}_p=\{T\in\mathcal{L}(H):\ {\rm Tr}(|T|^p)<\infty\}.$$
We set 
$$\|T\|_p=\big({\rm Tr}(|T|^p)\big)^{\frac1p},\quad T\in\mathcal{L}_p.$$
For every $p>0,$ $\|\cdot\|_p$ is a quasi-norm\footnote{A quasinorm satisfies the norm axioms, except that the triangle inequality is replaced by $||x+y||\leq K(||x||+||y||)$ for some uniform constant $K>1.$} and $(\mathcal{L}_p,\|\cdot\|_p)$ is a quasi-Banach space. For $p\geq1,$ $\|\cdot\|_p$ is a norm. For $p<1,$ the space $(\mathcal{L}_p,\|\cdot\|_p)$ is not Banach --- that is, its quasi-norm is not equivalent to any norm.

For a given $0<p\leq\infty,$ we let $\mathcal{L}_{p,\infty}$ denote the principal ideal in $\mathcal{L}(H)$ generated by the operator ${\rm diag}(\{(k+1)^{-\frac1p}\}_{k\geq0}).$ Equivalently,
$$\mathcal{L}_{p,\infty}=\{T\in\mathcal{L}(H): \mu(k,T)=O((k+1)^{-1/p})\}.$$
We set
$$\|T\|_{p,\infty}=\sup_{k\geq0}(k+1)^{1/p}\mu(k,T),\quad T\in\mathcal{L}_{p,\infty}.$$
For every $p>0,$ $\|\cdot\|_{p,\infty}$ is a quasi-norm and $(\mathcal{L}_{p,\infty},\|\cdot\|_{p,\infty})$ is a quasi-Banach space. For $p>1,$ $\|\cdot\|_{p,\infty}$ is equivalent to a (unitarily invariant Banach) norm. For $p\leq1,$ the space $(\mathcal{L}_{p,\infty},\|\cdot\|_{p,\infty})$ is not Banach --- that is, its quasi-norm is not equivalent to any norm. In \cite{Pietsch2009}, the Banach envelope of $\mathcal{L}_{1,\infty}$ was thoroughly investigated.

\subsection{Traces on $\mathcal{L}_{1,\infty}$}

\begin{defi}\label{trace def} If $\mathcal{I}$ is an ideal in $\mathcal{L}(H),$ then a unitarily invariant linear functional $\varphi:\mathcal{I}\to\mathbb{C}$ is said to be a trace.
\end{defi}

Since $U^{-1}TU-T=[U^{-1},TU]$ for all $T\in\mathcal{I}$ and for all unitaries $U\in\mathcal{L}(H),$ and since the unitaries span $\mathcal{L}(H),$ it follows that traces are precisely the linear functionals on $\mathcal{I}$ satisfying the condition
$$\varphi(TS)=\varphi(ST),\quad T\in\mathcal{I}, S\in\mathcal{L}(H).$$
The latter may be reinterpreted as the vanishing of the linear functional $\varphi$ on the commutator subspace which is denoted $[\mathcal{I},\mathcal{L}(H)]$ and defined to be the linear span of all commutators $[T,S]:\ T\in\mathcal{I},$ $S\in\mathcal{L}(H).$ It is shown in Lemma 5.2.2 in \cite{LSZ} that $\varphi(T_1)=\varphi(T_2)$ whenever $0\leq T_1,T_2\in\mathcal{I}$ are such that the singular value sequences $\mu(T_1)$ and $\mu(T_2)$ coincide.

For $p>1,$ the ideal $\mathcal{L}_{p,\infty}$ does not admit a non-zero trace \cite{DFWW}, while for $p=1,$ there exists a plethora of traces on $\mathcal{L}_{1,\infty}$ (see e.g. \cite{SSUZ-pietsch} or \cite{LSZ}). A standard example of a trace on $\mathcal{L}_{1,\infty}$ is a Dixmier trace introduced in \cite{Dixmier} that we now explain.

\begin{defi} Let $\omega$ be a free ultrafilter on $\mathbb{Z}_+.$ The functional
$${\rm Tr}_{\omega}:A\to\lim_{n\to\omega}\frac1{\log(2+n)}\sum_{k=0}^n\mu(k,A),\quad 0\leq A,$$
is finite and additive on the positive cone of $\mathcal{L}_{1,\infty}.$ Therefore, it extends to a trace on $\mathcal{L}_{1,\infty}.$ We call such traces Dixmier traces. 
\end{defi}

These traces clearly depend on the choice of the ultrafilter $\omega$ on $\mathbb{Z}_+.$ Using a slightly different definition, this notion of trace was applied by Connes \cite{Connes} in noncommutative geometry.

An extensive discussion of traces, and more recent developments in the theory, may be found in \cite{LSZ} including a discussion of the following facts. We refer the reader to an alternative approach to the theory of traces on $\mathcal{L}_{1,\infty}$ suggested in \cite{SSUZ-pietsch} (based on the fundamental paper \cite{Pietsch-nach} by Pietsch).
\begin{enumerate}
\item All Dixmier traces on $\mathcal{L}_{1,\infty}$ are positive.
\item All positive traces on $\mathcal{L}_{1,\infty}$ are continuous in the quasi-norm topology.
\item There exist positive traces on $\mathcal{L}_{1,\infty}$ which are not Dixmier traces (see \cite{SSUZ-pietsch}).
\item There exist traces on $\mathcal{L}_{1,\infty}$ which fail to be continuous (see \cite{LSZ}).
\end{enumerate}

\begin{defi} We say that an operator $A\in\mathcal{L}_{1,\infty}$ is measurable if $\varphi(A)$ does not depend on the choice of the continuous normalised trace $\varphi$ on $\mathcal{L}_{1,\infty}.$
\end{defi}

\subsection{Noncommutative plane: algebra} Each assertion in this subsection is rigorously established in Section 6 in \cite{LeSZ-cwikel}.

Our approach to the noncommutative plane is to introduce the von Neumann algebra generated by a strongly continuous family of unitary operators $\{U(t)\}_{t\in\mathbb{R}^d}$, $d\in\mathbb{N},$ satisfying the commutation relation 
\begin{equation}\label{nc_plane_comm_relation}
U(t+s) = \exp(-\frac{i}{2}\langle t,\theta s\rangle) U(t)U(s),\quad t,s \in \mathbb{R}^d,
\end{equation}
where $\theta$ is a fixed antisymmetric real $d\times d$ matrix. Namely, we set 
\begin{equation}\label{nc_plane_realisation}
(U(t)\xi)(u)=e^{-\frac{i}{2}\langle t,\theta u\rangle}\xi(u-t),\quad \xi\in L_2(\mathbb{R}^d), \quad u,t\in \mathbb{R}^d.
\end{equation}

\begin{defi}\label{nc_plane_definition} Let $d\in\mathbb{N}$ and let $\theta$ be a fixed non-degenerate\footnote{A non-degenerate antisymmetric matrix is automatically of even order.} antisymmetric real $d\times d$ matrix. The von Neumann subalgebra in $\mathcal{L}(L_2(\mathbb{R}^d))$ generated by $\{U(t)\}_{t \in \mathbb{R}^d}$, introduced in \eqref{nc_plane_realisation}, is called the noncommutative plane and denoted by $L_\infty(\mathbb{R}^d_\theta).$
\end{defi}

\begin{ex} If $d=2,$ then $L_\infty(\mathbb{R}^d_\theta)$ is generated by $2$ unitary groups $t\to U_1(t),$ $t\to U_2(t),$ $t\in\mathbb{R}$ satisfying the condition
$$U_1(t_1)U_2(t_2)=e^{i\alpha t_1t_2}U_2(t_2)U_1(t_1),\quad t_1,t_2\in\mathbb{R}.$$
Here, $U_1(t_1)=U((t_1,0))$ and $U_2(t_2)=U((0,t_2)).$
\end{ex}

The following assertion is well-known. In \cite{LeSZ-cwikel}, a {\it spatial} isomorphism is constructed.

\begin{thm}\label{spatial} For every non-degenerate antisymmetric real matrix $\theta,$ the algebra $L_\infty(\mathbb{R}^d_\theta)$ is isomorphic to $\mathcal{L}(L_2(\mathbb{R}^{\frac{d}{2}})).$
\end{thm}

Having established the isomorphism between $r:L_\infty(\mathbb{R}^d_\theta)\to \mathcal{L}(L_2(\mathbb{R}^{\frac{d}{2}}))$ we now equip $L_\infty(\mathbb{R}^d_\theta)$ with a faithful normal semifinite trace $\tau_\theta={\rm Tr}\circ r.$

We can now define $L_p-$spaces on $L_\infty(\mathbb{R}^d_\theta).$ 
$$L_p(\mathbb{R}^d_\theta)=\Big\{x\in L_\infty(\mathbb{R}^d_\theta):\ \tau_{\theta}(|x|^p)<\infty\Big\}.$$

\begin{lem}\label{nc_plane_L_2_description}
An operator $x\in L_\infty(\mathbb{R}^d_\theta)$ is in $L_2(\mathbb{R}^d_\theta)$ if and only if\footnote{To be precise,
$$x=\lim_{N\to\infty}\frac{1}{(2\pi)^{d/4}}\int_{[-N,N]^d}f(s)U(s)ds,$$
where the limit is taken in $L_2(\mathbb{R}^d_{\theta}).$ In what follows, we write the integral over $\mathbb{R}^d$ instead of the limit in order to lighten the notations.
}
$$x={\rm Op}(f)\stackrel{def}{=}\frac{1}{(2\pi)^{d/4}}\int_{\mathbb{R}^d}f(s)U(s)ds$$
for some unique $f\in L_2(\mathbb{R}^d)$ with $\|x\|_2=\|f\|_2.$
\end{lem}

Note that our picture is the Fourier dual of the one considered in \cite{gayral-moyal}. More precisely, the paper \cite{gayral-moyal} deals with operators of the form ${\rm Op}(\mathcal{F}f),$ where $f$ is Schwartz (in \cite{gayral-moyal}, these operators are written simply as $f$). 

\subsection{Noncommutative plane: calculus} Each assertion in this subsection is rigorously established in Section 6 in \cite{LeSZ-cwikel}.

Let $D_k,$ $1\leq k\leq d$ be multiplication operators on  $L_2(\mathbb{R}^d)$
$$(D_k\xi)(t)=t_k\xi(t), \quad \xi\in L_2(\mathbb{R}^d).$$
For brevity, we denote $\nabla=(D_1,\cdots,D_d).$ For every $1\leq k\leq d,$ we have
\begin{equation}\label{dk action}
[D_k,U(s)]=s_kU(s),\quad s\in\mathbb{R}^d.
\end{equation}
Moreover, we have
\begin{equation}\label{exp nabla action}
e^{i\langle t,\nabla\rangle}U(s)e^{-i\langle t,\nabla\rangle}=e^{i\langle t,s\rangle}U(s),\quad s,t\in\mathbb{R}^d.
\end{equation}

If $[D_k,x]\in \mathcal{L}(L_2(\mathbb{R}^d))$ for some $x\in L_\infty(\mathbb{R}^d_\theta),$ then $[D_k,x]\in L_\infty(\mathbb{R}^d_\theta).$ This crucial fact allows us to introduce mixed partial derivative $\partial^\alpha x$ of $x\in L_\infty(\mathbb{R}^d_\theta).$

\begin{defi} Let $\alpha$ be a multiindex and let $x\in L_\infty(\mathbb{R}^d_\theta).$ If every repeated commutator $[D_{\alpha_j},[D_{\alpha_j+1},\cdots,[D_{\alpha_n},x]]],$ $1\leq j\leq n,$ is a bounded operator on $L_2(\mathbb{R}^d),$ then the mixed partial derivative  $\partial^\alpha x$ of $x$ is defined as 
$$\partial^{\alpha}x=[D_{\alpha_1},[D_{\alpha_2},\cdots,[D_{\alpha_n},x]]].$$
In this case, we have that $\partial^{\alpha} x\in L_\infty(\mathbb{R}_\theta^d).$ As usual, $\partial^0 x=x.$
\end{defi} 

Therefore, we can introduce the Sobolev space $W^{m,p}(\mathbb{R}_\theta^d)$ associated with the noncommutative plane in the following way.

\begin{defi} For $m\in\mathbb{Z}_+$ and $p\geq 1,$ the space $W^{m,p}(\mathbb{R}^d_\theta)$ is the space of $x\in L_p(\mathbb{R}^d_\theta)$ such that every partial derivative of $x$ up to order $m$ is also in $L_p(\mathbb{R}^d_\theta).$ This space is equipped with the norm,
$$\|x\|_{W^{m,p}}=\sum_{|\alpha|\leq m}\|\partial^\alpha x\|_p,\quad x\in W^{m,p}(\mathbb{R}^d_\theta).$$
\end{defi}

The following assertion is one of the main results in \cite{LeSZ-cwikel}.

\begin{thm}\label{ncplane cwikel} If $x\in W^{d,1}(\mathbb{R}^d_{\theta}),$ then
\begin{enumerate}[{\rm (a)}]
\item\label{cwika} $x(1-\Delta)^{-\frac{d+1}{2}}\in\mathcal{L}_1$ and
$$\|x(1-\Delta)^{-\frac{d+1}{2}}\|_1\leq c_d\|x\|_{W^{d,1}}.$$
\item\label{cwikb} $x(1-\Delta)^{-\frac{d}{2}}\in\mathcal{L}_{1,\infty}$ and
$$\|x(1-\Delta)^{-\frac{d}{2}}\|_{1,\infty}\leq c_d\|x\|_{W^{d,1}}.$$
\end{enumerate}
\end{thm}

\section{Integration formula modulo a constant factor}\label{proportionality section}

For every $\phi\in L_{\infty}(\mathbb{R}^d),$ we define a bounded operator $T_{\phi}:L_2(\mathbb{R}^d_{\theta})\to L_2(\mathbb{R}^d_{\theta})$ by the formula
$$T_{\phi}:\int_{\mathbb{R}^d}f(s)U(s)ds\to \int_{\mathbb{R}^d}f(s)\phi(s)U(s)ds,\quad f\in L_2(\mathbb{R}^d).$$

\begin{lem}\label{fourier multiplier lemma} If $\phi$ is a Schwartz function, then $T_{\phi}:L_1(\mathbb{R}^d_{\theta})\to L_1(\mathbb{R}^d_{\theta}).$
\end{lem}
\begin{proof} We claim that
$$T_\phi x=\int_{\mathbb{R}^d}(\mathcal{F}\phi)(u)U(-\theta^{-1}u)xU(\theta^{-1}u)du,\quad x\in L_2(\mathbb{R}^d_{\theta}).$$

Since both sides above define bounded operators on $L_2(\mathbb{R}^d_{\theta})$ and since the set $\{{\rm Op}(f):\ f\mbox{ is Schwartz}\}$ is dense in $L_2(\mathbb{R}^d_{\theta}),$ it suffices to establish the claim for 
$$x=\int_{\mathbb{R}^d}f(s)U(s)ds,\quad f\in\mathcal{S}(\mathbb{R}^d).$$
Using the inverse Fourier transform, we write
$$\phi(s)=\int_{\mathbb{R}^d}(\mathcal{F}\phi)(u)e^{i\langle u,s\rangle}du,\quad s\in\mathbb{R}^d.$$
Since both $f$ and $\mathcal{F}\phi$ are Schwartz functions, it follows that
$$T_{\phi}x=\iint_{\mathbb{R}^d\times\mathbb{R}^d}f(s)(\mathcal{F}\phi)(u)e^{i\langle u,s\rangle}U(s)dsdu.$$
It follows from \eqref{nc_plane_comm_relation} that
$$e^{i\langle u,s\rangle}U(s)=U(-\theta^{-1}u)U(s)U(\theta^{-1}u).$$
Therefore,
$$T_{\phi}x=\int_{\mathbb{R}^d}(\mathcal{F}\phi)(u)\Big(\int_{\mathbb{R}^d}f(s)U(-\theta^{-1}u)U(s)U(\theta^{-1}u)ds\Big)du.$$
Using the definition of $x,$ we obtain
$$\int_{\mathbb{R}^d}f(s)U(-\theta^{-1}u)U(s)U(\theta^{-1}u)ds=U(-\theta^{-1}u)xU(\theta^{-1}u).$$
This proves the claim.

Now, we prove the assertion of the lemma as follows.
$$\|T_\phi x\|_1\leq\int_{\mathbb{R}^d}|(\mathcal{F}\phi)(u)|\cdot\|U(-\theta^{-1}u)xU(\theta^{-1}u)\|_1du=\|\mathcal{F}\phi\|_1\|x\|_1.$$
\end{proof}

\begin{lem}\label{continuity lemma} For every $x\in W^{d,1}(\mathbb{R}^d_{\theta}),$ the mapping
$$t\to U(-t)xU(t),\quad t\in\mathbb{R}^d,$$
is a continuous $W^{d,1}(\mathbb{R}^d_{\theta})-$valued function. Moreover,
$$\|U(-t)xU(t)\|_{W^{d,1}}=\|x\|_{W^{d,1}}.$$
\end{lem}
\begin{proof} It follows from Leibniz rule that
$$[D_k,U(-t)xU(t)]=[D_k,U(-t)]\cdot xU(t)+U(-t)\cdot [D_k,x]\cdot U(t)+U(-t)x\cdot[D_k,U(t)]=$$
$$=-t_kU(-t)xU(t)+U(-t)[D_k,x]U(t)+t_kU(-t)xU(t)=U(-t)[D_k,x]U(t).$$
Iterating the latter inequality, we obtain
$$\partial^{\alpha}(U(-t)xU(t))=U(-t)\partial^{\alpha}(x)U(t).$$
Thus,
$$\|U(-t)xU(t)\|_{W^{d,1}}=\sum_{|\alpha|\leq d}\|\partial^{\alpha}(U(-t)xU(t))\|_1=$$
$$=\sum_{|\alpha|\leq d}\|U(-t)\partial^{\alpha}(x)U(t)\|_1=\sum_{|\alpha|\leq d}\|\partial^{\alpha}(x)\|_1=\|x\|_{W^{d,1}}.$$

We now establish the continuity. For every $y\in\mathcal{L}_1,$ the mapping
$$t\to V(-t)yV(t),\quad t\in\mathbb{R}^d,$$
is continuous in the $\mathcal{L}_1-$norm whenever the mapping $t\to V(t)$ is strongly continuous. Recall that $(L_{\infty}(\mathbb{R}^d_{\theta}),\tau_{\theta})$ is $*-$isomorphic (so that trace is preserved) to $(\mathcal{L}(L_2(\mathbb{R}^{\frac{d}{2}})),{\rm Tr}).$ Thus, the mapping
$$t\to U(-t)\partial^{\alpha}(x)U(t)=\partial^{\alpha}(U(-t)xU(t))$$
is continuous in $L_1-$norm. This completes the proof.
\end{proof}

\begin{lem}\label{density lemma} 
\begin{enumerate}[{\rm (a)}]
\item\label{densa} If $f$ is Schwartz, then ${\rm Op}(f)\in W^{d,1}(\mathbb{R}^d_{\theta}).$
\item\label{densb} The set $\{{\rm Op}(f):\ f\mbox{ is Schwartz}\}$ is dense in $L_1(\mathbb{R}^d_{\theta}).$ In particular, $W^{d,1}(\mathbb{R}^d_{\theta})$ is dense in $L_1(\mathbb{R}^d_{\theta}).$
\end{enumerate}
\end{lem}
\begin{proof} There exists a sequence $\{e_{kl}\}_{k,l\geq0}\subset L_{\infty}(\mathbb{R}^d_{\theta})$ such that
\begin{enumerate}[{\rm (i)}]
\item $e_{k_1l_1}e_{k_2l_2}=\delta_{l_1,k_2}e_{k_1l_2}$ and $e_{kl}^*=e_{lk}.$
\item $\tau_{\theta}(e_{kk})=1.$
\item $\sum_{k\geq0}e_{kk}=1$ in strong operator topology.
\item for every $k,l\geq0,$ there exists a Schwartz function $f_{kl}$ such that $e_{kl}={\rm Op}(f_{kl}).$
\end{enumerate}
The existence of such a sequence is established in Lemma 2.4 in \cite{gayral-moyal} (see also additional references therein). A particular formula for $f_{kl}$ can be found on p.~618 in \cite{gayral-moyal} in terms of Laguerre polynomials.

We prove \eqref{densa}. Let $f$ be a Schwartz function. By Proposition 2.5 in \cite{gayral-moyal}, one can write $f$ as 
$$f=\sum_{k,l\geq0}c_{kl}f_{kl},\quad \sum_{k,l\geq0}|c_{kl}|<\infty.$$
Thus,
$${\rm Op}(f)=\sum_{k,l\geq0}c_{kl}e_{kl},$$
where the series converges in $L_1-$norm. Thus, ${\rm Op}(f)\in L_1(\mathbb{R}^d_{\theta}).$ Let $f_{\alpha}(t)=t^{\alpha}f(t),$ $t\in\mathbb{R}^d.$ By \eqref{dk action}, $\partial^{\alpha}({\rm Op}(f))={\rm Op}(f_{\alpha}).$ Since $f_{\alpha}$ is also a Schwartz function, it follows that $\partial^{\alpha}({\rm Op}(f))\in L_1(\mathbb{R}^d_{\theta}).$ This proves \eqref{densa}.

To prove \eqref{densb}, note that, for every $x\in L_1(\mathbb{R}^d_{\theta}),$
$$\sum_{k,l\leq N}e_{kk}xe_{ll}=(\sum_{k\leq N}e_{kk})x(\sum_{l\leq N}e_{ll})\to x$$
in $\mathcal{L}_1-$norm as $N\to\infty.$ Note that $e_{kk}xe_{ll}$ is a scalar multiple of $e_{kl}={\rm Op}(f_{kl}).$ Since a linear combination of Schwartz functions is again a Schwartz function, it follows that
$$\sum_{k,l\leq N}e_{kk}xe_{ll}\in \{{\rm Op}(f):\ f\mbox{ is Schwartz}\}\subset W^{d,1}(\mathbb{R}^d_{\theta}).$$
This proves \eqref{densb}.
\end{proof}

\begin{lem}\label{main invariance lemma} If $F$ is a continuous functional on $W^{d,1}(\mathbb{R}^d_{\theta})$ such that
$$F(x)=F(U(-t)xU(t)),\quad x\in W^{d,1}(\mathbb{R}^d_{\theta}),\quad t\in\mathbb{R}^d,$$
then $F=\tau_{\theta}$ (up to a constant factor).
\end{lem}
\begin{proof} Let $T:W^{d,1}(\mathbb{R}^d_{\theta})\to W^{d,1}(\mathbb{R}^d_{\theta})$ be defined by setting
$$Tx=\int_{\mathbb{R}^d}U(-\theta^{-1}t)xU(\theta^{-1}t)e^{-\frac12|t|^2}dt.$$
The integral is understood as a Bochner integral of a continuous $W^{d,1}(\mathbb{R}^d_{\theta})-$valued function (the continuity and convergence of the integral follow from Lemma \ref{continuity lemma}).

For every $x\in W^{d,1}(\mathbb{R}^d_{\theta}),$ we have
$$F(Tx)=\int_{\mathbb{R}^d}F(U(-\theta^{-1}t)xU(\theta^{-1}t))e^{-\frac12|t|^2}dt=\int_{\mathbb{R}^d}F(x)e^{-\frac12|t|^2}dt=(2\pi)^{\frac{d}{2}}F(x).$$
Thus,
$$F(x)=(2\pi)^{-\frac{d}{2}}F(Tx),\quad x\in W^{d,1}(\mathbb{R}^d_{\theta}).$$

We claim that $\|Tx\|_{W^{d,1}}\leq c_d\|x\|_1$ for every $x\in W^{d,1}(\mathbb{R}^d_{\theta}).$ To see this, let
$$x=\int_{\mathbb{R}^d}f(s)U(s)ds,\quad f\in L_2(\mathbb{R}^d).$$
If, in the proof of Lemma \ref{fourier multiplier lemma}, we select $\phi(t)=e^{-\frac12|t|^2},$ $t\in\mathbb{R}^d,$ then the argument given there yields
$$Tx=\int_{\mathbb{R}^d}f(s)U(s)e^{-\frac12|s|^2}ds.$$
By \eqref{dk action}, we have
$$\partial^{\alpha}(Tx)=\int_{\mathbb{R}^d}f(s)U(s)s^{\alpha}e^{-\frac12|s|^2}ds.$$
Let $\phi_{\alpha}(s)=s^{\alpha}e^{-\frac12|s|^2},$ $s\in\mathbb{R}^d.$ We have that $\partial^{\alpha}\circ T=T_{\phi_{\alpha}}.$ By Lemma \ref{fourier multiplier lemma}, $T_{\phi_{\alpha}}:L_1(\mathbb{R}^d_{\theta})\to L_1(\mathbb{R}^d_{\theta})$ is a bounded operator. This proves the claim.

For every $x\in W^{d,1}(\mathbb{R}^d_{\theta}),$ we have
$$|F(x)|=(2\pi)^{-\frac{d}{2}}|F(Tx)|\leq(2\pi)^{-\frac{d}{2}}\|F\|_{(W^{d,1})^*}\|Tx\|_{W^{d,1}}\leq c_d\|F\|_{(W^{d,1})^*}\|x\|_1.$$
Thus, a functional $F$ on $W^{d,1}(\mathbb{R}^d_{\theta})$ is bounded in $\|\cdot\|_1-$norm. By the Hahn-Banach Theorem, $F$ extends to a bounded functional on $L_1(\mathbb{R}^d_{\theta}).$ Hence, there exists $y\in L_{\infty}(\mathbb{R}^d_{\theta})$ such that
$$F(x)=\tau_{\theta}(xy),\quad x\in W^{d,1}(\mathbb{R}^d_{\theta}).$$
Clearly,
$$F(U(-t)xU(t))=\tau_{\theta}(U(-t)xU(t)y)=\tau_{\theta}(xU(t)yU(-t)).$$
Comparing the last 2 equalities, we obtain
$$\tau_{\theta}(xU(t)yU(-t))=\tau_{\theta}(xy),\quad x\in W^{d,1}(\mathbb{R}^d_{\theta}).$$
Since $W^{d,1}(\mathbb{R}^d_{\theta})$ is dense in $L_1(\mathbb{R}^d_{\theta}),$ it follows that $y=U(t)yU(-t)$ for every $t\in\mathbb{R}^d.$ In other words, $y$ commutes with every $U(t)$ and, therefore, with every element in $L_{\infty}(\mathbb{R}^d_{\theta}).$ Since $L_{\infty}(\mathbb{R}^d_{\theta})$ is a factor (see Theorem \ref{spatial}), it follows that $y$ is a scalar operator. This completes the proof.
\end{proof}

The following proposition is a light version of Theorem \ref{cif ncplane}.

\begin{prop}\label{proportionality prop} If $x\in W^{d,1}(\mathbb{R}^d_{\theta}),$ then $x(1-\Delta)^{-\frac{d}{2}}\in\mathcal{L}_{1,\infty}$ and
$$\varphi(x(1-\Delta)^{-\frac{d}{2}})=c_{\varphi}\tau_{\theta}(x)$$
for every continuous trace on $\mathcal{L}_{1,\infty}$ and for some constant $c_{\varphi}.$ 
\end{prop}
\begin{proof} By Theorem \ref{ncplane cwikel} \eqref{cwikb}, the functional
$$F:x\to \varphi(x(1-\Delta)^{-\frac{d}{2}}),\quad x\in W^{d,1}(\mathbb{R}^d_{\theta}),$$
is a well defined bounded linear functional on $W^{d,1}(\mathbb{R}^d_{\theta}).$

Since $\varphi$ is unitarily invariant, it follows that
$$\varphi(x(1-\Delta)^{-\frac{d}{2}})=\varphi(e^{i\langle t,\nabla\rangle}x(1-\Delta)^{-\frac{d}{2}}e^{-i\langle t,\nabla\rangle}),\quad t\in\mathbb{R}^d.$$
By the Spectral Theorem, we have
$$(1-\Delta)^{-\frac{d}{2}}e^{-i\langle t,\nabla\rangle}=e^{-i\langle t,\nabla\rangle}(1-\Delta)^{-\frac{d}{2}},$$
and so
$$\varphi(x(1-\Delta)^{-\frac{d}{2}})=\varphi(e^{i\langle t,\nabla\rangle}xe^{-i\langle t,\nabla\rangle}(1-\Delta)^{-\frac{d}{2}}).$$

For every $s\in\mathbb{R}^d,$ we have (see \eqref{exp nabla action})
$$e^{i\langle t,\nabla\rangle}U(s)e^{-i\langle t,\nabla\rangle}=e^{i\langle t,s\rangle}U(s).$$
On the other hand, it follows from \eqref{nc_plane_comm_relation} that
$$U(-\theta^{-1}t)U(s)U(\theta^{-1}t)=e^{i\langle t,s\rangle}U(s).$$
Comparing preceding equalities, we arrive at
$$e^{i\langle t,\nabla\rangle}U(s)e^{-i\langle t,\nabla\rangle}=U(-\theta^{-1}t)U(s)U(\theta^{-1}t).$$
It follows that
$$e^{i\langle t,\nabla\rangle}xe^{-i\langle t,\nabla\rangle}=U(-\theta^{-1}t)xU(\theta^{-1}t),\quad x\in L_{\infty}(\mathbb{R}^d_{\theta}).$$

Combining the preceding paragraphs, we obtain
$$\varphi(x(1-\Delta)^{-\frac{d}{2}})=\varphi(U(-\theta^{-1}t)xU(\theta^{-1}t)(1-\Delta)^{-\frac{d}{2}}).$$
Applying Lemma \ref{main invariance lemma} to our functional $F,$ we conclude the argument.
\end{proof}

\section{Proof of measurability}\label{measurability section}

\begin{lem}\label{l1 estimate} If $K\in W^{2d+2,1}([0,1]^d\times[0,1]^d)$ and if $T:L_2((0,1)^d)\to L_2((0,1)^d)$ is an integral operator with integral kernel $K,$ then $T\in\mathcal{L}_1$ and $\|T\|_1\leq c_d\|K\|_{W^{2d+2,1}}.$
\end{lem}
\begin{proof} Let $K\in W^{2d+2,1}([-\pi,\pi]^d\times[-\pi,\pi]^d)$ be an extension of $K$ such that
$$\|K\|_{W^{2d+2,1}([-\pi,\pi]^d\times[-\pi,\pi]^d)}\leq c_d\|K\|_{W^{2d+2,1}([0,1]^d\times[0,1]^d)}$$
and such that $K$ vanishes on and near the boundary. Thus, $K\in W^{2d+2,1}(\mathbb{T}^d\times\mathbb{T}^d).$ Let $S:L_2(\mathbb{T}^d)\to L_2(\mathbb{T}^d)$ be an integral operator with integral kernel $K.$ We have $T=M_{\chi_{(0,1)^d}}SM_{\chi_{(0,1)^d}}.$ Thus, $\|T\|_1\leq\|S\|_1.$

Let us write Fourier series
$$K(t,s)=\sum_{m_1,m_2\in\mathbb{Z}^d}c_{m_1,m_2}e_{m_1}(t)e_{m_2}(s),\quad t,s\in\mathbb{T}^d.$$
Set
$$S_{m_1,m_2}\xi=\langle\xi,e_{-m_2}\rangle e_{m_1},\quad\xi\in L_2(\mathbb{T}^d).$$
It is an integral operator on $L_2(\mathbb{T}^d)$ with the integral kernel $(t,s)\to e_{m_1}(t)e_{m_2}(s).$ Hence,
$$S=\sum_{m_1,m_2\in\mathbb{Z}^d}c_{m_1,m_2}S_{m_1,m_2}.$$
By triangle inequality, we have
$$\|S\|_1\leq\sum_{m_1,m_2\in\mathbb{Z}^d}|c_{m_1,m_2}|\leq$$
$$\leq\sup_{m_1,m_2\in\mathbb{Z}^d}(1+|m_1|^2+|m_2|^2)^{d+1}|c_{m_1,m_2}|\cdot\sum_{m_1,m_2\in\mathbb{Z}^d}(1+|m_1|^2+|m_2|^2)^{-d-1}.$$

Observe that $(1+|m_1|^2+|m_2|^2)^{d+1}c_{m_1,m_2}$ is the $(m_1,m_2)-$th Fourier coefficient of the function $(1-\Delta_{\mathbb{T}^{2d}})^{d+1}(K)$ (here, $\Delta_{\mathbb{T}^{2d}}$ is the Laplacian on the torus $\mathbb{T}^{2d}$). Taking into account that Fourier coefficients do not exceed the $L_1-$norm, we infer that
$$(1+|m_1|^2+|m_2|^2)^{d+1}|c_{m_1,m_2}|\leq(2\pi)^{-2d}\|(1-\Delta_{\mathbb{T}^{2d}})^{d+1}K\|_1\leq c_d\|K\|_{W^{2d+2,1}}.$$
Here, the last inequality follows from the definition of a Sobolev space.
\end{proof}

In what follows, we consider the tensor product of 2 bounded operators on a Hilbert space $H$ as a bounded operator on the Hilbert space $H\bar{\otimes}H.$

\begin{lem}\label{fubini lemma} If $T\in\mathcal{L}_{1,\infty}$ and $S\in\mathcal{L}_1,$ then $S\otimes T\in\mathcal{L}_{1,\infty}$ and
\begin{equation}\label{fubini formula}
\varphi(S\otimes T)={\rm Tr}(S)\cdot\varphi(T)
\end{equation}
for every continuous trace $\varphi$ on $\mathcal{L}_{1,\infty}.$
\end{lem}
\begin{proof} Firstly, we show that $S\otimes T\in\mathcal{L}_{1,\infty}.$ Let $z(t)=t^{-1},$ $t>0.$ By definition, we have $\mu(T)\leq\|T\|_{1,\infty}z.$ The crucial fact that $\mu(S\otimes z)=\|S\|_1z$ is proved on p. 211 in \cite{LT2}. Thus,
$$\|S\otimes T\|_{1,\infty}=\|S\otimes\mu(T)\|_{1,\infty}\leq\|T\|_{1,\infty}\|S\otimes z\|_{1,\infty}=\|T\|_{1,\infty}\|S\|_1.$$

We now turn to the proof of \eqref{fubini formula}. If $S$ is a rank one projection, then there is nothing to prove. If $S$ is a positive finite rank operator, then the assertion follows by linearity. If $S$ is an arbitrary finite rank operator, then the assertion again follows by linearity. 

Let $S\in\mathcal{L}_1$ be arbitrary. Fix $\epsilon>0$ and choose $S_1,S_2\in\mathcal{L}_1$ such that $S=S_1+S_2,$ $S_1$ is finite rank and $\|S_2\|_1\leq\epsilon.$ Clearly,
$$\varphi(S\otimes T)-{\rm Tr}(S)\cdot\varphi(T)=$$
$$=(\varphi(S_1\otimes T)-{\rm Tr}(S_1)\cdot\varphi(T))+(\varphi(S_2\otimes T)-{\rm Tr}(S_2)\cdot\varphi(T)).$$
By the preceding paragraph, the summand in the first bracket vanishes. Thus,
$$\varphi(S\otimes T)-{\rm Tr}(S)\cdot\varphi(T)=\varphi(S_2\otimes T)-{\rm Tr}(S_2)\cdot\varphi(T).$$
Hence,
$$|\varphi(S\otimes T)-{\rm Tr}(S)\cdot\varphi(T)|\leq|\varphi(S_2\otimes T)|+|{\rm Tr}(S_2)\cdot\varphi(T)|\leq$$
$$\leq\|\varphi\|_{\mathcal{L}_{1,\infty}^*}\cdot(\|S_2\otimes T\|_{1,\infty}+|{\rm Tr}(S_2)|\|T\|_{1,\infty}).$$
By the norm estimate in the first paragraph and by the assumption on $S_2,$ we have
$$|\varphi(S\otimes T)-{\rm Tr}(S)\cdot\varphi(T)|\leq 2\epsilon\|\varphi\|_{\mathcal{L}_{1,\infty}^*}\|T\|_{1,\infty}.$$
Since $\epsilon>0$ is arbitrarily small, the assertion follows.
\end{proof}

In the following lemma, we consider the direct sum of bounded operators on a Hilbert space $H$ as a bounded operator on a Hilbert space $\bigoplus_{m\geq0}H.$

\begin{lem}\label{direct sum lemma} If the operators $\{T_m\}_{m\geq0}$ are pairwise orthogonal, i.e. $T_{m_1}T_{m_2}=T_{m_1}^*T_{m_2}=0$ for $m_1\neq m_2,$ then $\sum_{m\geq0}T_m$ is unitarily equivalent\footnote{To be pedantic, $\sum_{m\geq0}T_m$ is unitarily equivalent to the direct sum $\bigoplus_{m\geq0}T_m|_{r_m(H)\to r_m(H)},$ where $r_m$ is the projection defined in the proof of Lemma \ref{direct sum lemma}. Clearly, $T_m$ is unitarily equivalent to the direct sum $T_m|_{r_m(H)\to r_m(H)}\bigoplus 0_{(1-r_m)(H)\to (1-r_m)(H)}.$ Thus, a direct sum $\bigoplus_{m\geq0}T_m$ is unitarily equivalent to $(\sum_{m\geq0}T_m)\bigoplus 0.$ In what follows, we ignore this subtle difference and write unitary equivalence as stated in Lemma \ref{direct sum lemma}.} to $\bigoplus_{m\geq0}T_m.$ Here, the sums are taken in the weak operator topology.
\end{lem}
\begin{proof} Let $p_1$ and $p_2$ be projections on $H.$ Since $t\to t^{\frac1n},$ $t>0,$ is an operator monotone function for every $n\geq1,$ it follows that
$$p_1=p_1^{\frac1n}\leq(p_1+p_2)^{\frac1n}\stackrel{{\rm sot}}{\to}{\rm supp}(p_1+p_2).$$
Similarly, $p_2\leq {\rm supp}(p_1+p_2)$ and, therefore,
$$p_1\vee p_2\leq {\rm supp}(p_1+p_2).$$
This simple fact can be also found in Proposition 2.5.14 in \cite{KR1}.

Let $p_m={\rm supp}(T_m)$ and $q_m={\rm supp}(T_m^*).$ It follows from the assumption that $p_{m_1}p_{m_2}=p_{m_1}q_{m_2}=q_{m_1}q_{m_2}=0,$ $m_1\neq m_2.$ Set $r_m=p_m\vee q_m.$ We have
$$(p_{m_1}+q_{m_1})(p_{m_2}+q_{m_2})=0,\quad m_1\neq m_2.$$
Thus,
$${\rm supp}(p_{m_1}+q_{m_1})\cdot{\rm supp}(p_{m_2}+q_{m_2})=0,\quad m_1\neq m_2.$$
By the preceding paragraph, we have $r_{m_1}r_{m_2}=0,$ $m_1\neq m_2.$

If $T=\sum_{m\geq0}T_m,$ then $r_mT=T_m$ and $Tr_m=T_m$ for every $m\geq0.$ Thus, $T=\bigoplus_{m\geq0}T_m,$ where $T_m$ acts on the Hilbert space $r_m(H).$
\end{proof}

Let
$$h(t)=(1+\sum_{k=1}^d\lfloor t_k\rfloor^2)^{-\frac{d}{2}},\quad t\in\mathbb{R}^d.$$

The following proposition yields a special case of Theorem \ref{cif ncplane}.

\begin{prop}\label{measurability prop} If $f$ is a Schwartz function supported on $[-1,1]^d$ and if $x={\rm Op}(f),$ then $xh(\nabla)$ is measurable.
\end{prop}
\begin{proof} {\bf Step 1:} We have that $xh(\nabla)$ is an integral operator with the kernel
$$K:(t,s)\to f(t-s)h(s)e^{\frac{i}{2}\langle s,\theta t\rangle},\quad t,s\in\mathbb{R}^2.$$
By assumption on $f,$ we have that
$$f(s-t)=0,\quad s\in m_1+[0,1]^d,\quad t\in m_2+[0,1]^2,\quad m_1-m_2\notin\{-1,0,1\}^d.$$
Thus,
$$xh(\nabla)=\sum_{l_1,l_2\in\{-1,0,1\}^d}T_{l_1,l_2},\quad T_{l_1,l_2}=\sum_{\substack{m\in\mathbb{Z}^d \\ m=l_2{\rm mod}3}}h(m)T_{m,l_1},$$
where $T_{m,l_1}$ is an integral operator whose integral kernel is given by the formula
$$(t,s)\to f(t-s)e^{\frac{i}{2}\langle s,\theta t\rangle}\chi_{m+l_1+[0,1]^d}(t)\chi_{m+[0,1]^d}(s),\quad t,s\in\mathbb{R}^d,$$

{\bf Step 2:} We claim that $T_{l_1,l_2}\in\mathcal{L}_{1,\infty}$ and is measurable.

Note that the operators $\{T_{m,l_1}\}_{\substack{m\in\mathbb{Z}^d \\ m=l_2{\rm mod}3}}$ are pairwise orthogonal. Therefore, we have ($\sim$ denotes unitary equivalence)
$$T_{l_1,l_2}\sim\bigoplus_{\substack{m\in\mathbb{Z}^d \\ m=l_2{\rm mod}3}}(1+|m|^2)^{-\frac{d}{2}}T_{m,l_1}.$$

By definition, $T_{m,l_1}:L_2(m+[-1,2]^d)\to L_2(m+[-1,2]^d).$ Define a unitary operator 
$$U_m:L_2([-1,2]^d)\to L_2(m+[-1,2]^d)$$
by setting
$$(U_m\xi)(t)=e^{\frac{i}{2}\langle m,\theta t\rangle}\xi(t-m),\quad\xi\in L_2([-1,2]^d),\quad t\in m+[-1,2]^d.$$
Define an operator $S_{l_1}:L_2([-1,2]^d)\to L_2([-1,2]^d)$ to be an integral operator with the integral kernel
$$(t,s)\to f(t-s)e^{\frac{i}{2}\langle s,\theta t\rangle}\chi_{l_1+[0,1]^d}(t)\chi_{[0,1]^d}(s),\quad t,s\in [-1,2]^d.$$
A direct computational argument shows that\footnote{Indeed,
$$(U_m^{-1}\xi)(t)=e^{-\frac{i}{2}\langle m,\theta t\rangle}\xi(t+m),\quad\xi\in L_2(m+[-1,2]^d),\quad t\in [-1,2]^d.$$
Thus,
$$(S_{l_1}U_m^{-1}\xi)(t)=\chi_{l_1+[0,1]^d}(t)\cdot\int_{[0,1]^d}f(t-s)e^{\frac{i}{2}\langle s,\theta(t+m)\rangle}\xi(s+m)ds.$$
Thus,
$$(U_mS_{l_1}U_m^{-1}\xi)(t)=\chi_{l_1+[0,1]^d}(t-m)\cdot\int_{[0,1]^d}e^{\frac{i}{2}\langle m,\theta t\rangle}f(t-s-m)e^{\frac{i}{2}\langle s,\theta t\rangle}\xi(s+m)ds=$$
$$=\chi_{m+l_1+[0,1]^d}(t)\cdot\int_{m+[0,1]^d}f(t-s)e^{\frac{i}{2}\langle s,\theta t\rangle}\xi(s)ds.$$
}
$$T_{m,l_1}=U_mS_{l_1}U_m^{-1}.$$

Hence,
$$T_{l_1,l_2}\sim\bigoplus_{\substack{m\in\mathbb{Z}^d \\ m=l_2{\rm mod}3}}(1+|m|^2)^{-\frac{d}{2}}S_{l_1}\sim S_{l_1}\otimes\Big\{(1+|m|^2)^{-\frac{d}{2}}\Big\}_{\substack{m\in\mathbb{Z}^d \\ m=l_2{\rm mod}3}}.$$
By Lemma \ref{l1 estimate}, $S_{l_1}\in\mathcal{L}_1.$ The claim follows now from Lemma \ref{fubini lemma}.
\end{proof}

\begin{proof}[Proof of Theorem \ref{cif ncplane}] Choose a Schwartz function $f_0$ supported on $[-1,1]^d$ such that $f_0(0)\neq0$ and let $x_0={\rm Op}(f_0).$ Set 
$$k(t)=(1+|t|^2)^{\frac{d+1}{2}}\cdot\Big((1+|t|^2)^{-\frac{d}{2}}-(1+\sum_{k=1}^d\lfloor t_k\rfloor^2)^{-\frac{d}{2}}\Big),\quad t\in\mathbb{R}^d.$$
Clearly, $k$ is a bounded function on $\mathbb{R}^d.$

By Lemma \ref{density lemma} \eqref{densa}, we have $x_0\in W^{d,1}(\mathbb{R}^d_{\theta}).$ Using the obvious equality
$$x_0(1-\Delta)^{-\frac{d}{2}}-x_0h(\nabla)=x_0(1-\Delta)^{-\frac{d+1}{2}}\cdot k(\nabla)$$
and Theorem \ref{ncplane cwikel} \eqref{cwika}, we infer that
$$x_0(1-\Delta)^{-\frac{d}{2}}-x_0h(\nabla)\in\mathcal{L}_1.$$
By Proposition \ref{measurability prop}, we have that $x_0h(\nabla)$ is measurable and, hence, so is the operator $x_0(1-\Delta)^{-\frac{d}{2}}.$

Let now $x\in W^{d,1}(\mathbb{R}^d_{\theta})$ be arbitrary. Since $f_0$ is a Schwartz function, it follows that
$$\tau_{\theta}(x_0)=f_0(0)\neq0.$$
Without loss of generality, $\tau_{\theta}(x_0)=1.$ Let $z=x-\tau_{\theta}(x)x_0\in W^{d,1}(\mathbb{R}^d_{\theta}).$ Clearly, $\tau_{\theta}(z)=0.$ We have
$$\varphi(x(1-\Delta)^{-\frac{d}{2}})=\varphi(z(1-\Delta)^{-\frac{d}{2}})+\tau_{\theta}(x)\cdot\varphi(x_0(1-\Delta)^{-\frac{d}{2}}).$$
By Proposition \ref{proportionality prop}, the first summand vanishes. By the preceding paragraph, the second summand does not depend on $\varphi.$ This completes the proof.
\end{proof}


\begin{thebibliography}{99}
\bibitem{BF} Benameur M., Fack T. {\it Type II non-commutative geometry. I. Dixmier trace in von Neumann algebras.} Adv. Math. {\bf 199} (2006), no. 1, 29--87.
\bibitem{CGRS1} Carey A., Gayral V., Rennie A., Sukochev F. {\it Integration on locally compact noncommutative spaces.} J. Funct. Anal. {\bf 263} (2012), no. 2, 383--414.
\bibitem{CGRS2} Carey A., Gayral V., Rennie A., Sukochev F. {\it Index theory for locally compact noncommutative geometries.} Mem. Amer. Math. Soc. {\bf 231} (2014), no. 1085. 
\bibitem{Connes} Connes A. {\it Noncommutative Geometry.} Academic Press, San Diego, 1994.
\bibitem{action} Connes A. {\it The action functional in noncommutative geometry.} Comm. Math. Phys. {\bf 117} (1988), no. 4, 673--683.
\bibitem{Dixmier} Dixmier J. {\it Existence de traces non normales.} (French)  C. R. Acad. Sci. Paris Ser. A-B  {\bf 262}  (1966) A1107--A1108.
\bibitem{DFWW} Dykema K., Figiel T., Weiss G., Wodzicki M. {\it Commutator structure of operator ideals.} Adv. Math. {\bf 185} (2004), no. 1, 1--79.
\bibitem{gayral-moyal} Gayral V., Gracia-Bondia J., Iochum B., Sch\"ucker T., Varilly J. {\it Moyal planes are spectral triples.} Comm. Math. Phys. {\bf 246} (2004), no. 3, 569--623.
\bibitem{GVF} Gracia-Bondia J., Varilly J., Figueroa H. {\it Elements of noncommutative geometry.} Birkhauser Advanced Texts: Basel Textbooks, Birkhauser Boston, Inc., Boston, MA, 2001.
\bibitem{KR1} Kadison R., Ringrose J. {\it Fundamentals of the theory of operator algebras. Vol. I. Elementary theory.} Reprint of the 1983 original. Graduate Studies in Mathematics, {\bf 15}. American Mathematical Society, Providence, RI, 1997.
\bibitem{KLPS} Kalton N., Lord S., Potapov D., Sukochev F. {\it Traces of compact operators and the noncommutative residue.} Adv. Math. {\bf 235} (2013), 1--55.
\bibitem{LeSZ-cwikel} Levitina G., Sukochev F., Zanin D. {\it Cwikel estimates revisited.} submitted manuscript.
\bibitem{LT2} Lindenstrauss J., Tzafriri L. {\it Classical Banach spaces. II. Function spaces.} Ergebnisse der Mathematik und ihrer Grenzgebiete, {\bf 97}. Springer-Verlag, Berlin-New York, 1979.
\bibitem{LPS} Lord S., Potapov D., Sukochev F. {\it  Measures from Dixmier traces and zeta functions.} J. Funct. Anal. {\bf 259} (2010), no. 8, 1915--1949.
\bibitem{LSZ} Lord S., Sukochev F., Zanin D. {\it Singular traces. Theory and applications.} De Gruyter Studies in Mathematics, {\bf 46}, De Gruyter, Berlin, 2013. 
\bibitem{Pietsch-nach} Pietsch A. {\it Traces and shift invariant functionals.} Math. Nachr. {\bf 145} (1990), 7--43.
\bibitem{Pietsch2009} Pietsch A. {\it About the Banach envelope of $l_{1,\infty}.$} Rev. Mat. Complut. {\bf 22} (2009), no. 1, 209--226.
\bibitem{SSUZ-pietsch} Semenov E., Sukochev F., Usachev A., Zanin D. {\it Banach limits and traces on $\mathcal{L}_{1,\infty}.$} Adv. Math. {\bf 285} (2015), 568--628.
\end{thebibliography}
\end{document}